\documentclass{amsart}

\let\oldmarginpar\marginpar
\renewcommand\marginpar[1]{\-\oldmarginpar[\raggedleft\footnotesize #1]%
{\raggedright\footnotesize #1}}

\usepackage{amssymb}
\usepackage{amsmath}
\usepackage{graphicx}
\usepackage{epstopdf}
\usepackage{rotating}

\usepackage{amscd}
\usepackage{graphics}
\usepackage{young}
\newtheorem{Theorem}{Theorem}[section]
\newtheorem{Proposition}[Theorem]{Proposition}
\newtheorem{Corollary}[Theorem]{Corollary}
\newtheorem{Lemma}[Theorem]{Lemma}

\def\proof{\par{\it Proof}. \ignorespaces}
\def\endproof{{\ \vbox{\hrule\hbox{%
     \vrule height1.3ex\hskip0.8ex\vrule}\hrule }}\par}
\newenvironment{Proof}{\proof}{\endproof}

\theoremstyle{definition}
\newtheorem{Definition}{Definition}[section]
\newtheorem{Example}{Example}[section]

\newtheorem{Conjecture}{Conjecture}[section]

\theoremstyle{remark}

\numberwithin{equation}{section}

\let\trueint=\int
\let\truesum=\sum
\def\int{\mathop{\textstyle\trueint}\limits}
\def\sum{\mathop{\textstyle\truesum}\limits}

\let\<=\langle
\let\>=\rangle
\def\Real{{\mathbb{R}}}

\def\Gr{{\rm Gr}}
\def\SL{{\rm SL}}

\begin{document}
 
\title[Cohomology of real Grassmannians]
{On the cohomology of real Grassmann manifolds}

\author{Luis Casian}
\address{Department of Mathematics, Ohio State University, Columbus,
OH 43210}
\email{casian@math.ohio-state.edu}
\author{Yuji Kodama$^*$}
\thanks{$^*$Partially
supported by NSF grant DMS-1108813.}

\address{Department of Mathematics, Ohio State University,
Columbus, OH 43210}
\email{kodama@math.ohio-state.edu}


\begin{abstract}
We give an explicit and simple construction of  the incidence graph for the integral
cohomology of {\it real} Grassmann manifold $\Gr(k,n)$
in terms of the Young diagrams filled with the letter $q$ in checkered pattern. 
It turns out that there are two types of graphs, one for the trivial coefficients and other for the twisted
coefficients, and they compute the homology groups of the orientable and non-orientable cases of $\Gr(k,n)$ via the Poincar\'e-Verdier duality.
We also give  an explicit formula of the Poincar\'e polynomial
for $\Gr(k,n)$ and show that the Poincar\'e polynomial is also related to the number of points
on $\Gr(k,n)$ over a finite field $\mathbb{F}_q$ with $q$ being a power of prime
which is also used in the Young diagrams. 
\end{abstract}

\maketitle

\thispagestyle{empty}
\pagenumbering{arabic}\setcounter{page}{1}
\tableofcontents


\section{Introduction}
The Grassmann manifolds (sometimes, referred to as simply the Grassmannian) which  parametrize vector subspaces of  fixed dimensions  of a given vector space are
 fundamental objects, and appear in many
areas of mathematics. Their study  dates back to Pl\"ucker  who, in the 19th century,  considered  the case  of  vector subspaces
of dimension $k=2$ of a space of dimension $n=4$.  In this note we study the integral cohomology  of {\em real} Grassmann manifolds.
The cohomology of {\em complex} Grassmann manifolds is well-known but the real case poses additional challenges. We give
an explicit  description of the Betti numbers and the incidence graph which leads to the integral cohomology.

There is a well-known  decomposition of a real Grassmann manifold into {\em Schubert cells}. The Schubert  cells can be  parametrized  with Young diagrams
and these Young diagrams will then also appear in a co-chain complex to calculate the cohomology.  In this note  the Young diagrams are filled in 
with a letter $q$ in a checkered pattern and this gives rise to a power of $q$ for each Young diagram and therefore to a polynomial by considering certain \emph{alternating} sum. This polynomial is computed explicitly 
and, it is shown to  contain the information of all the Betti numbers of the real Grassman manifolds.  That is, it gives the Poincar\'e polynomial of the manifold.
There is also an alternative construction of these polynomials  in terms of  varieties over a finite field ${\mathbb F}_q$
where this time $q$ is a power of a prime number.  One can consider  varieties over  ${\mathbb F}_q$ that naturally correspond to  the {\em real} Graassmann manifolds and then count
the number of  ${\mathbb F}_q$ points. This again gives rise to  a polynomial in $q$  which agrees
with the polynomial that was obtained directly from  the Young diagrams.

The results in this paper were initially motivated by considering real Grassmann manifolds in terms of the KP hierarchy \cite{CK10}.

\section{The real Grassmann manifold $\Gr(k,n)$}
The real Grassmannian $\Gr(k,n)$ is defined by 
\[
\Gr(k,n):=\{\text{the set of~all~}k\text{-dimensional subspaces in}~ \mathbb{R}^n\}.
\]


Let $\{e_j:j=1,\ldots,n\}$ be a basis of $\mathbb{R}^n$, and $\{f_i:i=1,\ldots,k\}$
a basis of $k$-dimensional subspace $V_k(e_1,\ldots,e_n)$. Then there exists a full rank $k\times n$ matrix
$A$ such that we have
\[
(f_1,\ldots,f_k)=(e_1,\ldots, e_n) A^T,
\]
where $A^T$ is the transpose of the matrix $A$.  Since the left action of $\text{GL}_k(\mathbb{R})$
on $A$ does not change the subspace $V_k$, $\Gr(k,n)$ can be expressed as
\[
\Gr(k,n)\cong \text{GL}_k(\mathbb{R})\setminus M_{k\times n}(\mathbb{R}),
\]
where $M_{k\times n}(\mathbb{R})$ is the set of all $k\times n$ matrices of rank $k$.
The $\text{GL}_k(\mathbb{R})$-action puts $A$ in the canonical form called the reduced row echelon form ({RREF}).  For example, the RREF for the generic matrix has the form,
\[
A=\begin{pmatrix}
* & \cdots & * & 1 &  \cdots &0\\
\vdots & \ddots &\vdots&\vdots&\ddots&\vdots\\
*&\cdots&*&0&\cdots&1
\end{pmatrix} ~\in~\Gr(k,n).
\]
The {1}'s in the matrix are called the pivots of the RREF. In general, each $k$-element subset $I=\{i_1,\ldots,i_k\}$ of $[n]=\{1,2,\ldots,n\}$ is called the {\it Schubert symbol} and represents the index set of the pivots for the matrix $A$.  Then the set of all $k$-element subsets  $\binom{[n]}{k}$ gives a parametrization for the Schubert decomposition of $\Gr(k,n)$, i.e.
\[
\Gr(k,n)=\bigsqcup_{I\in\binom{[n]}{k}}\Omega^\circ_{I}
\]
where $\Omega^\circ_I$ is the Schubert cell, and $A\in \Omega^\circ_I$ means that the RREF of $A$ has the pivot index $I$.

\subsection{The Schubert decomposition of $\Gr(k,n)$ and the Bruhat poset $S_n^{(k)}$}
Let $S_n^{(k)}$ denote the set of minimal coset representatives defined by (see \cite{bjorner:05})
\[
S_n^{(k)}:=\{t\in S_n:~\ell(ts)\ge \ell(t),~\forall s\in P_k\},
\]
where $\ell(s)$ is the length of $s$, and $P_k$ is the maximal parabolic subgroup 
generated  by 
\[
P_k:=\langle s_1,\ldots,\hat{s}_k,\ldots, s_{n-1}\rangle \cong S_k\times S_{n-k}.
\]
Since  there is a bijection between $\binom{[n]}{k}$ and $S_n^{(k)}$,
the Schubert decomposition can be also expressed in terms of $S_n^{(k)}$, i.e.
\[
\Gr(k,n)=\bigsqcup_{w\in S_n^{(k)}}\Omega^\circ_w,
\]
where we identify $\Omega_I^\circ=\Omega^\circ_w$.  More precisely, we have $
I=\{i_1,i_2,\ldots,i_k\}=w\cdot\{1,2,\ldots,k\}$, i.e. $i_j=w(j)$.
Note that the dimension of $\Omega^\circ_w$ is given by dim$(\Omega^\circ_w)=\ell(w)$.

We label the Schubert cell $\Omega^\circ_w$ for $w\in S_n^{(k)}$ by the Young diagram $Y_w$. Given a Schubert symbol $\{w(1)<\cdots<w(k)\}$, first define
$\lambda:=(\lambda_1,\ldots,\lambda_r)$ with $\lambda_j=w(j)-j$
(note $\lambda_1\le\cdots\le \lambda_r$) and $|\lambda|:=\sum_{j=1}^r\lambda_j$.  This $\lambda$ gives a partition of number $|\lambda|$.  Then the Young digram with shape
$\lambda$ is a collection of $|\lambda|$ top-left-justified boxes 
with $\lambda_j$ boxes in the $j$-th row from the bottom as shown below:
\setlength{\unitlength}{0.54mm}
\begin{center}
  \begin{picture}(130,65)
  
\put(5,55){\line(1,0){145}}
\put(5,45){\line(1,0){112}}
  \put(5,35){\line(1,0){91}}
  \put(5,25){\line(1,0){91}}
  \put(5,15){\line(1,0){71}}
  \put(5,5){\line(1,0){71}}
  \put(5,5){\line(0,1){50}}
  \put(20,5){\line(0,1){50}}
  \put(35,5){\line(0,1){50}}
  \put(76,5){\line(0,1){10}}
  \put(60,5){\line(0,1){10}}
  \put(96,25){\line(0,1){10}}
  \put(80,25){\line(0,1){10}}
\put(117,45){\line(0,1){10}}
\put(100,45){\line(0,1){10}}
\put(50,45){\line(0,1){10}}
  \put(10,29){$s_{j}$}
  \put(22,29){$s_{j+1}$}
  \put(-5,29){$j$}
  \put(98,29){$i_{j}=w(j)$}
  \put(81,29){$s_{i_{j}-1}$}
  \put(10,49){$s_k$}
  \put(-5,49){$k$}
  \put(22,49){$s_{k+1}$}
  \put(-5,38){$\vdots$}
  \put(-5,18){$\vdots$}
  \put(65,38){$\vdots$}
  \put(101,49){$s_{i_k-1}$}
  \put(67,49){$\cdots$}
  \put(45,29){$\cdots$}
  \put(20,20){}
  \put(47,18){$\vdots$}
  \put(10,9){$s_1$}
  \put(-5,9){$1$}
  \put(45,9){$\cdots$}
  \put(78,9){$i_1=w(1)$}
  \put(25,9){$s_2$}
  \put(61,9){$s_{i_1-1}$}
    \put(119,48){$i_k=w(k)$}
  \put(5,59){$k+1$}
 \put(21,59){$k+2$}
  \put(100,59){$k+\lambda_k~\quad\cdots~ \quad n$}
      \put(67,59){$\cdots$}
\put(11,-2){$1$}  
\put(26,-2){$2$}    
\put(45,-2){$\cdots$}
\put(61,-2){$i_1-1$}
 \end{picture} 
\end{center}

\smallskip

In the diagram, one can consider a lattice path starting from the bottom left corner and ending to
the top right conner with the label $\{1,2,\ldots,n\}$ in the counterclockwise direction,  then the indices in the Schubert symbol $I=\{w(1)<\cdots<w(k)\}$ appear at the vertical paths. In the case of $I=\{n-k+1,\ldots,n\}$, i.e. the top cell, corresponding to the longest element $w_\circ\in S_n^{(k)}$,
we write $Y_{w_\circ}=(n-k)^k$ representing the $k\times(n-k)$ rectangular diagram.

Here a reduced expression of $w$ can be written by 
\[
w=w_1w_2\cdots w_k \quad{\rm with}\quad w_j:=s_{i_{j}-1}\cdots s_{j},\quad j=1,\ldots,k,
\]
where each $w_j$ is the product of $s_l$ in the $j$-th row from the bottom.
Note that each  $Y_w$ is a sub-diagram of the $k\times (n-k)$ of the top cell, and we have the inclusion relation,
\[
Y_w~\subseteq~Y_{w'}\qquad\Longleftrightarrow\qquad w~\le~ w'.
\]
(Notice the direction of the Bruhat order in this paper.)   We then define:
\begin{Definition} Based on a weak Bruhat order, we define the {\it weak Bruhat graph},
\[
\mathsf{B}(k,n):=\{( Y_w, \to):w\in S_n^{(k)}\}.
\]
That is, the vertices are gives by the Young diagrams $Y_w$, and the edges are defined by
\[
Y_w~\longrightarrow~ Y_{w'}\qquad\text{if}\quad w'=s_jw\quad\text{with}\quad\ell(w')=\ell(w)+1,
\]
for some $s_j\in S_n$.  This graph is also referred to as {\it Young's lattice}
of $\Gr(k,n)$.


Below we define additional graphs  $\mathsf{G}(k,n)$, $\mathsf{G}^*(k,n)$ whose vertices have parameters involving  the Young diagrams $Y_w$. The blanket notational convention  is that if there is an edge in any of these graphs involving a pair $w, w'$ then there is a simple reflection $s_j$ such that  $w'=s_jw$ with $\ell(w')=\ell(w)+1$. These graphs  can therefore be considered as subgraphs of the weak Bruhat graph even if this is not explicitly restated. 
\end{Definition}

\begin{Example} \label{gr25} The graph $\mathsf{B}(2,5)$ is given by
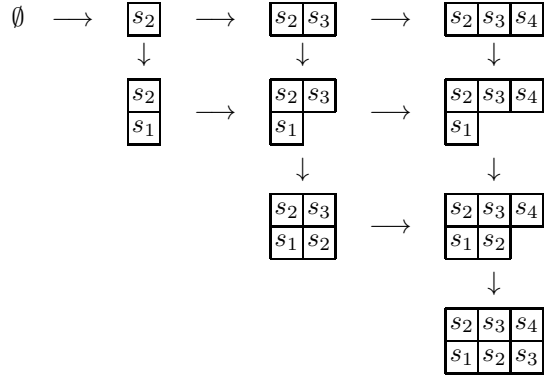
\[
\begin{matrix}
\emptyset  &\longrightarrow&     \young[1][8][$s_2$]      & \longrightarrow   &     \young[2][8][$s_2$,$s_3$]        & \longrightarrow & \young[3][8][$s_2$,$s_3$,$s_4$]    \\[-2.5ex]\\
             &                                &\downarrow     &                               &        \downarrow  &                          &\downarrow \\[-2ex]\\
            &                               &    \young[1,1][8][$s_2$,$s_1$]        &  \longrightarrow    &  \young[2,1][8][$s_2$,$s_3$,$s_1$]           & \longrightarrow & \young[3,1][8][$s_2$,$s_3$,$s_4$,$s_1$]   \\[-2ex]  \\
            &                           &                                      &                                &   \downarrow    &                              &\downarrow \\[-2ex]\\
            &                            &                                   &                                &       \young[2,2][8][$s_2$,$s_3$,$s_1$,$s_2$]       &\longrightarrow  &\young[3,2][8][$s_2$,$s_3$,$s_4$,$s_1$,$s_2$] \\[-2ex]\\
            &                          &                                     &                               &                                        &                         &\downarrow \\[-2ex]\\
            &                          &                                  &                                  &                                       &                         & \young[3,3][8][$s_2$,$s_3$,$s_4$,$s_1$,$s_2$,$s_3$]
  \end{matrix}
\]
Each arrow $\to$ represents the action $s_j$ to the pivot set, and new box added by this action is labeled by $s_j$.  For example, $s_2: (1,2)\to (1,3)$ in the top left one. 
Notice that a reduced expression of $w\in S_n^{(k)}$ can be obtained by following the paths $\to$ starting from $\emptyset$ to the diagram $Y_w$.
For example,  the diagram at the right side in the second row can be written by
\[
w=s_1s_4s_3s_2=s_4s_3s_1s_2=s_4s_1s_3s_2.
\]
depending on the three different paths.
\end{Example}

\section{The integral cohomology of $\Gr(k,n)$ and the incidence graph}

Let us first denote the Schubert variety marked by $Y_w$ as
\[
\Omega_w:=\bigcup_{w'\le w}\Omega^\circ_w.
\]
Let $\sigma_w$ denote the Schubert class associated to $\Omega_w$,
i.e. $\sigma_w:=[\Omega_w]$.
Then the  co-chain complex of $\Gr(k,n)$ is given in the form,
\[
\mathcal{C}^*=\bigoplus_{j=0}^{k(n-k)}\mathcal{C}^j,\quad {\rm with}\quad 
\mathcal{C}^j=\bigoplus_{\ell(w)=j}\mathbb{Z} ~\sigma_w,
\]
The co-boundary operator,
$\delta_j:\mathcal{C}^j\to\mathcal{C}^{j+1}$, gives the form,
\[
\delta_j\sigma_w=\sum_{{w'}>{w}}~{[w;w']}~\sigma_{w'},
\]
where the coefficient $[w;w']$ is called the {\it incidence} number and it takes either ${0}$ or ${\pm 2}$ (see e.g. \cite{koch,casian99}).
We denote the incidence graph of $\Gr(k,n)$ by
\begin{equation}\label{incidenceG}
\mathcal{G}(k,n):=\left\{(\sigma_w,\Rightarrow):\sigma_w\Rightarrow \sigma_{w'}
~\text{if}~[w;w']=\pm 2\right\}.
\end{equation}

Then identifying $\sigma_w$ with the Young diagram $Y_w$ in the Bruhat
graph $\mathsf{B}(k,n)$,  the {\it arrow} $\to$ in $\mathsf{B}(k,n)$ can be considered as the action of the  co-boundary operator $\delta_j$.
The incidence graph $\mathcal{G}(k,n)$ is then obtained by giving
specific information on the incidence numbers for those {arrows}.  The main result of this paper is to provide a ``{\it simple}'' representation of the incidence graph in terms of the Bruhat graph with those specific incidence numbers:
For this purpose, let us first define:

\begin{Definition} \label{weighted}
We define a ${q}$-{\it weighted} Schubert variety $\langle Y_w; \eta(w)\rangle$
where $\eta(w)$ is defined as follows:


\begin{itemize}
\item[(1)]  Insert the letter ${q}$ and ${1}$ alternatively in the boxes of the Young diagram $Y_w$, such that  ${q}$ locates at the northwest conner box of the diagram.  That is, the first box with $s_k$ has
$q$, then the boxes with $s_{k\pm 2j}$ have $q$, and the boxes with $s_{k\pm(2j+1)}$ have 1.
\item[(2)]  Then $\eta(w)$ for each diagram $Y_w$ filled with $1$ and $q$ in the previous step 
is defined by
\[
\eta(w)=\text{the total number of~}q\text{'s~in~the~diagram~} Y_w.
\]
\end{itemize}
We will also need to refer to the $q$ weighted Schubert variety $\langle Y_w; \eta^*(w)\rangle$ defined in exactly  the same way but  
now the  boxes with $s_{k\pm 2j}$ have $1$, and the boxes with $s_{k\pm(2j+1)}$ have $q$.  To make a distinction, the first  arrangement of $q$'s
in the Young diagram will be referred to as {\em standard} and the second as  $q$-{\em shifted}.  We will show in subsection \ref{proof} below that the monomial $q^{\eta(w)}$ can be expressed a
in terms of the Hecke algebra action, a $q$-deformation of the Weyl group $W$.
\end{Definition}

 For each $w\in S_n^{(k)}$, one can easily find the formula
 of $\eta(w)$.
\begin{Lemma}\label{etaformula} The function $\eta(w)$ is given by
\[
\eta(w)=\sum_{j=1}^k\left\lfloor\frac{w(j)-j+\sigma(j)}{2}\right\rfloor
\]
where $\sigma(j)=1$ if $k-j=$even, and $\sigma(j)=0$ if $k-j=$odd.
\end{Lemma}

\begin{Definition}\label{graphG}  We define a graph $\mathsf{G}(k,n)$ for the $q$-weighted Schubert varieties by modifying the Bruhat graph $\mathsf{B}(k,n)$,
\[
\mathsf{G}(k,n)=\{(\langle Y_w;\eta(w)\rangle,\Rightarrow): Y_w\Rightarrow Y_{w'}\text{~if~}\eta(w)=\eta(w')\}.
\]
That is, identifying each Young diagram with $q$'s as the $q$-weighted Schubert variety, and replacing the edge $\to$ between $ Y_w$ and $Y_{w'}$ in the Bruhat graph $\mathsf{B}(k,n)$ with ${\Rightarrow}$ if $\eta(w)=\eta(w')$.  

There is an alternative graph $\mathsf{G}^*(k,n)$ in which an edge $\Rightarrow$ corresponds instead to  $\eta^*(w)=\eta^*(w')$, i.e.
\[
\mathsf{G}^*(k,n)=\{(\langle Y_w;\eta^*(w)\rangle,\Rightarrow): Y_w\Rightarrow Y_{w'}\text{~if~}\eta^*(w)=\eta^*(w')\}.
\]  
\end{Definition}

\begin{Example} The graph $\mathsf{G}(2,5)$ is given by

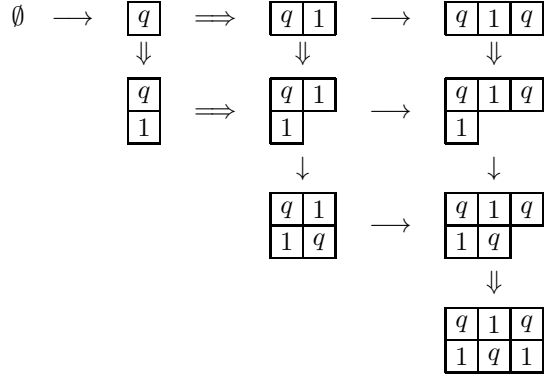
\[
\begin{matrix}
\emptyset  &  \longrightarrow  &     \young[1][8][$q$]      & {\Longrightarrow}   &     \young[2][8][$q$,$1$]        &\longrightarrow & \young[3][8][$q$,$
1$,$q$]    \\[-2.5ex]\\
             &                                &{\Downarrow }    &                               &       {\Downarrow } &                          &{\Downarrow} \\[-2ex]\\
            &                               &    \young[1,1][8][$q$,$1$]        & {\Longrightarrow }   &  \young[2,1][8][$q$,$1$,$1$]           & \longrightarrow & \young[3,1][8][$q$,$1$,$q$,$1$]   \\[-2ex]  \\
            &                           &                                      &                                &  \downarrow   &                              &\downarrow \\[-2ex]\\
            &                            &                                   &                                &       \young[2,2][8][$q$,$1$,$1$,$q$]       &\longrightarrow &\young[3,2][8][$q$,$1$,$q$,$1$,$q$] \\[-2ex]\\
            &                          &                                     &                               &                                        &                         &{\Downarrow} \\[-2ex]\\
            &                          &                                  &                                  &                                       &                         & \young[3,3][8][$q$,$1$,$q$,$1$,$q$,$1$]
  \end{matrix}
\]
\end{Example}

Then we have the main theorem of this paper:

\begin{Theorem}\label{main}
The graph $\mathsf{G}(k,n)$ with
the edges $\Rightarrow$ having the incidence number $\pm 2$ (others are zero)
gives the {\it incidence} graph $\mathcal{G}(k,n)$ of the real Grassmann manifold $\Gr(k,n)$.
\end{Theorem}

Based on this theorem, we can compute the $\mathbb{Z}$-cohomology of $\Gr(k,n)$.
For example, from the above example, one can easily compute the ${\mathbb{Z}}$-cohomology of $\Gr(2,5)$, and we have
\begin{align*}
H^0&=\mathbb{Z},\quad H^1=0,\quad H^2=\mathbb{Z}_2,\quad H^3=\mathbb{Z}_2\\
H^4&=\mathbb{Z}\oplus\mathbb{Z}_2,\quad H^5=0,\quad H^6=\mathbb{Z}_2
\end{align*}
Note $\Gr(2,5)$ is {not} orientable. {This is always true if $M$ is odd}, that is, we have:
\begin{Corollary}
$\Gr(k,n)$ is orientable {\it iff} $n$ is even.
\end{Corollary}
\begin{Proof} Since the box at the southeast conner of the top Young diagram, $(n-k)^k$, has 
the letter ${q}$ iff $n$ is {\it even}.  This implies the incidence number $[w;w_{\circ}]=0$ for
the longest element $w_{\circ}$.
\end{Proof}


\subsection{Outline of the proof of the main theorem}\label{proof}

The proof  of Theorem \ref{main} relies on the description given in \cite{casian:06} on the incidence graph for real flag manifolds. 
In \cite{casian:06}, each local system  on the real flag manifold is parametrized in terms of a vector of signs $\epsilon=(\epsilon_1,\cdots , \epsilon_{n-1})\in\{\pm\}^{n-1}$,
which encodes the structure of the local system (see Definition 4.3 of \cite{casian:06}). The sign $\epsilon_i=-$ indicates that the local system is  {\em constant}along $s_i$ (restricted to a fiber 
which corresponds to the flag manifold of an $\text{SL}_2({\mathbb R})$). Similarly, the sign $\epsilon_i=+$ corresponds to
the local system which is  {\em  twisted}  along $s_i$.  In the case of cohomology with constant coefficients, we define an action of the Weyl group on the signs starting with $(-,\cdots ,-)$ representing a {\it constant} local system:

\begin{Definition}\label{WactionS} \cite[Definition 2.10]{casian:06}  Each simple reflection $s_i\in W$  acts on the sign $\epsilon_j$ in the vector $\epsilon=(\epsilon_1,\ldots,\epsilon_{n-1})\in \{\pm\}^{n-1}$ as
\[
s_i:~\epsilon_j ~\longrightarrow~\epsilon_j\epsilon_i^{-C_{j,i}},
\]
where $C_{j,i}$ is the Cartan matrix of $\mathfrak{sl}_n(\mathbb{R})$.  That is,
$s_i(\epsilon_i)=\epsilon_i, s_i(\epsilon_{i\pm1})=\epsilon_{i\pm1}\epsilon_i$ and $s_i(\epsilon_j)=\epsilon_j$ for $|i-j|\ge 2$.
\end{Definition}

With the $W$-action on $\{\pm\}^{n-1}$, we also define a number ${\hat{\eta}(w,\epsilon)}$ 
with $\hat{\eta}(e,\epsilon)=0$ by
\[
\hat\eta(s_iw,\epsilon)~=\left\{\begin{array}{lll}
\hat\eta(w,\epsilon)\qquad&\text{if}\quad s_i(w(\epsilon))=w(\epsilon),\\[1.0ex]
1+\hat\eta(w,\epsilon)  \qquad&\text{if}\quad s_i(w(\epsilon))\ne w( \epsilon).
\end{array}\right.
\]
  The $ \hat\eta(w,\epsilon)$  is then  simply the number of times that the simple reflections
   in a fixed  reduced expression of $w$
 have acted non-trivially (changed a sign) starting with  $\epsilon=(\epsilon_1,\cdots,\epsilon_{n-1})$ (see
 Definition 4.6 and Proposition 4.2 of  \cite{casian:06}).
  The numbers $\hat{\eta}(w,\epsilon)$ have the algebraic description in
 terms of the Hecke algebra $\mathcal{H}=\mathbb{Z}[q,q^{-1}]\otimes_{\mathbb{Z}}\mathbb{Z}[W]$ (see Section 4 of \cite{casian:06} for the details).  

We now recall the connection with the incidence graph $\mathcal{G}(k,n)$.  There is an edge $\Rightarrow$   between two vertices,
$\sigma_w\Rightarrow \sigma_{s_iw}$ related by a simple reflection $s_i$ exactly when the $s_i$-action does not change the signs, i.e. $w(\epsilon)=s_i  w(\epsilon)$ starting with $\epsilon=(-,\ldots,-)$.  Then defining a weighted Schubert class by $\langle \sigma_w;\hat{\eta}(w,\epsilon)\rangle$, we have
the following theorem which is proven in \cite{casian99}, and also in Theorem 4.9 in \cite{casian:06}:
\begin{Theorem}  With $\epsilon=(-,\ldots,-)$, 
the incidence graph $\mathcal{G}(k,n)$ is expressed by
\[
\mathcal{G}(k,n)=\left\{(\langle\sigma_w;\hat{\eta}(w,\epsilon)\rangle,\Rightarrow):\sigma_w\Rightarrow\sigma_{w'}~
\text{if}~\hat{\eta}(w,\epsilon)=\hat{\eta}(w',\epsilon)\right\},
\]
that is, one can replace $[w;w']\ne 0$ with the equation $\hat{\eta}(w,\epsilon)=\hat{\eta}(w',\epsilon)$.
\end{Theorem}


For $\Gr(k,n)$ there are exactly two initial signs $\epsilon$ corresponding to local systems on $\Gr(n,k)$.   Let $\epsilon_{\pm}$ denote those initial signs which are defined by
\[
\epsilon_{\pm}:=(-,\ldots,-,\overset{k}{\pm},-,\ldots,-).
\]
As explained above, $\epsilon_-$ represents a constant local system, while
$\epsilon_+$ represents a non-constant local system.
Then we have:
\begin{Lemma} \label {lema2}  For any $w\in S_n^{(k)}$, one can show that
\begin{enumerate}
\item[(a)]  $w(\epsilon_{-})= s_iw( \epsilon_{-})$  if and only if  $\eta(w)= \eta(s_iw)$, and
\item[(b)]  $w(\epsilon_{+})= s_iw (\epsilon_{+})$  if and only if  $ \eta^*(w)= \eta^*(s_iw)$.
\end{enumerate}
\end{Lemma}
\begin{proof} We give an outline of the proof by induction on $n$  using
the weak Bruhat graph $\mathsf{B}(k,n)$ (see Example \ref{gr25}).
We  first consider the partial graph consisting of the Young diagrams $Y_w$ with one row, i.e. $w=s_{k+m-1}\cdots s_{k+1}s_k$, or in terms of the Schubert symbol $w\cdot\{1,\ldots,k\}$ which is
\[
\{1,2,\ldots,k-1,k+m\}\qquad\text{for}\quad m=0,\ldots,n-k.
\]
In terms of the signs representing local systems,  we compute $w(\epsilon_-)$ with $\epsilon_-=(-,\ldots,-)$.
It is easy to check with a direct calculation that  Lemma \ref{lema2} holds for this
partial graph. One recovers the incidence graph of $\Gr(1, n-k+1)\cong \mathbb{R}P^{n-k}$ in this first step. This graph  has an alternating pattern in $\Rightarrow$. 
This edges $\Rightarrow$ can be simultaneously  described in terms of agreement of $\eta(w)$ and in terms of the $W$-action on signs. Lemma \ref{lema2} is then
true for the edges in this partial graph. 

If one starts with $\epsilon_{+}$ then the graph
 starts with $\Rightarrow$ and this corresponds to a $q$-shifted arrangement which is the the incidence graph of $\Gr(1,n-k+1)$ with local coefficients.


We can  now decompose the graph $\mathsf{B}(k,n)$ into a disjoint union
of sub-graphs.  Each sub-graph consists of the Young diagrams with the same
first row having $m$ boxes for $m=1,\ldots,n-k$. This sub-graph can be identified with
$\mathsf{B}(k-1,k+m-1)$ after removing the first common top row (in the Example \ref{gr25}, those subgraphs correspond
to the columns of $\mathsf{B}(2,5)$).
We note, however,  that if start with a Young diagram with a {\em standard}  $q$ arrangement,   the Young diagrams in the decomposition, excluding that the common top row, correspond to $\Gr(k-1, k+m-1)$ but are $q$-{\em shifted}.
This is a trivial observation related to the fact that $q$'s are placed in an alternating pattern and we are starting from the second row which now
has $q$ and not $1$ at the beginning.  
We thus have a decomposition into a disjoint union of subgraphs associated to $\Gr(k-1,k+m-1)$. Then by induction,  Lemma \ref{lema2} holds for the subgraphs and this allows us to complete the proof.
\end{proof}

This Lemma with the definition of $\hat{\eta}(w,\epsilon)$ then leads to the following Proposition:
\begin{Proposition}\label{lema1}
We have $\hat \eta (w, \epsilon_{-}) = \eta(w)$ and $\hat \eta (w, \epsilon_{+}) = \eta^*(w)$.
\end{Proposition}
This Proposition gives the proof of Theorem \ref{main}.  That is, we 
have $\mathcal{G}(k,n)=\mathsf{G}(k,n)$ by identifying
$\langle\sigma_w;\hat{\eta}(w,\epsilon_-)\rangle$ with $\langle Y_w;\eta(w)\rangle$.  Also, for $\mathsf{G}^*(k,n)$, we have the incidence graph $\mathcal{G}^*(k,n)$ corresponding
to a non-constant local system, i.e.
\[
\mathcal{G}^*(k,n):=\left\{(\langle\sigma_w;\hat{\eta}(w,\epsilon_+)\rangle,\Rightarrow):
\sigma_w\Rightarrow\sigma_{w'}~\text{if}~\hat{\eta}(w,\epsilon_+)=\hat{\eta}(w'\epsilon_+)\right\}.
\]

 As the simplest but important example (used in the proof as the first step), we give the incidence graph of $\Gr(1,n)\cong\Real P^{n-1}$ whose the longest element is $w_\circ=s_{n-1}\cdots s_2s_1$. The graph $\mathsf{G}(1,n)$ is given by
\[
\emptyset~\overset{s_1}{\longrightarrow}~\young[1][8][$q$]~\overset{s_2}{\Longrightarrow}~\young[2][8][$q$,$1$]~\overset{s_3}{\longrightarrow}~\young[3][8][$q$,$1$,$q$]\quad{\Longrightarrow}\quad\cdots
\]
Recall that the boxes marked by $s_m$ with odd $m$ has $q$ and others have $1$.
Then we recover the well-known results of the integral cohomology of $\Real P^{n-1}$,
\[
H^k(\Real P^{n-1};\mathbb{Z})=\left\{\begin{array}{lll}
\mathbb{Z}\quad&\text{if}\quad k=0~\text{or}~n-1=\text{odd}\\
\mathbb{Z}_2\quad&\text{if}\quad k=\text{even}\\
0\quad &\text{otherwise}\\
\end{array}\right.
\]
Also the graph $\mathsf{G}^*(1,n)$ is given by
\[
\emptyset~\overset{s_1}{\Longrightarrow}~\young[1][8][$1$]~\overset{s_2}{\longrightarrow}~\young[2][8][$1$,$q$]~\overset{s_3}{\Longrightarrow}~\young[3][8][$1$,$q$,$1$]\quad{\longrightarrow}\quad\cdots
\]
Then we have the cohomology with twisted coefficient $\mathcal{L}$,
\[
H^k(\Real P^{n-1};\mathcal{L})=\left\{\begin{array}{lll}
\mathbb{Z}\quad&\text{if}\quad k=n-1=\text{even}\\
\mathbb{Z}_2\quad&\text{if}\quad k=\text{odd}\\
0\quad &\text{otherwise}\\
\end{array}\right.
\]
As we will show in Section \ref{Poincare}, this cohomology is related to the
homology of the non-orientable case $\mathbb{R}P^{n-1}$ with $n=$odd
via the Poincar\'e-Verdier duality, i.e.
\[
H_k(\mathbb{R}P^{n-1};\mathbb{Z})=H^{n-1-k}(\mathbb{R}P^{n-1};\mathcal{L}).
\]

\section{The $\mathbb{F}_q$-points on $\Gr(k,n)$}

Let $\mathbb{F}_q$ be a finite field with $q$ elements. Here $q$  is a power of a prime.
Then based on the $q$-weighted Schubert variety defined in the previous section, one can find the number of $\mathbb{F}_q$ points
on $\Gr(k,n)$.  We first recall that in the complex case, the $\mathbb{F}_q$-points of $\Gr(k,n,\mathbb{C})$ is given by
\begin{equation}\label{Qnumber}
|\Gr(k,n,\mathbb{F}_q)|:=\sum_{w\in S_n^{(k)}}q^{\ell(w)}=\left[\begin{matrix}n\\k\end{matrix}\right]_q:=\frac{[n]_q!}{[k]_q![n-k]_q!}
\end{equation}
where $[k]_q!:=[k]_q[k-1]_q\cdots[1]_q$ and $[k]_q$ is the $q$-analog of $k$,
\[
[k]_q:=\frac{1-q^k}{1-q}=1+q+\cdots+q^{k-1},
\]
Then the $t^2$-binomial coefficient $\left[\begin{matrix}n\\k\end{matrix}\right]_{t^2}$ gives the Poincar\'e polynomial $P^{\mathbb{C}}_{(k,n)}(t)$ of $\Gr(k,n;\mathbb{C})$.  


\subsection{The number of $\mathbb{F}_q$-points on $\Gr(k,n)$}\label{Fqpoints}

The real Grassmanian  $\Gr(k,n)$ can then be described as $K/ L\cap K$, where $K={\rm SO}_n(\Real)$  and
the Levi factor $L$ of a maximal parabolic subgroup $P=P_k$
containing the Borel subgroup $B$  of upper triangular matrices in $\SL_n({\mathbb R})$. We can consider  $K/ L\cap K$
over different fields such as the algebraically closed fields fields such as  ${\mathbb C}, \overline {\mathbb F}_q$ or the finite field $ {\mathbb F}_q$.
 This allows us to consider  
$\Gr(k,n)_{\mathbb{F}_q}$ and the number of  points  $|\Gr(k,n)_{\mathbb{F}_q}|$.  We refer to as the $\mathbb{F}_q$-points on $\Gr(k,n)$. Here we give an explicit formula
of the $\mathbb{F}_q$-points
in terms of $\eta(w)$ describing the number of $q$'s in the Young diagram $Y_w$.
Let us first give:
\begin{Definition}\cite[Definition 3.1]{casian:06}\label{polyQ}
We define a polynomial $p(q)$ by
\[
p(q)=(-1)^{k(n-k)}\sum_{w\in S_n^{(k)}}(-1)^{\ell(w)}q^{{\eta(w)}}.
\]
\end{Definition}
Then from tyne incidence graph $\mathsf{G}(k,n)$, one can find the explicit formula of the polynomial $p(q)$.
\begin{Proposition}
For $\Gr(k,n)$, the polynomial $p(q)$ takes the following form,
\begin{itemize}
\item[(i)] if $(k,n)$ equals $(2j,2m), (2j,2m+1)$ or $(2j+1,2m+1)$, then
\[
p(q)=\left[\begin{matrix}m\\j\end{matrix}\right]_{q^2},
\]
\item[(ii)] if $(k,n)=(2j+1,2m)$, then
\[
p(q)=(q^m-1)\left[\begin{matrix}m-1\\j\end{matrix}\right]_{q^2}.
\]
\end{itemize}.
\end{Proposition}
\begin{proof}  
Let $Y_w$ be a Young diagram satisfying the following condition:
\begin{itemize}
\item[(i)] If $(k,n)$ equals $(2j,2m), (2j,2m+1)$ or $(2j+1,2m+1)$, then 
 $Y_w$ consists of only the sub-diagram given by
\[
\young[2,2][8][$q$,$1$,$1$,$q$].
\]
\item[(ii)] If $(k,n)=(2j+1,2m)$, $Y_w$
consists of the above sub-diagram {\it{and}} the hook diagram, $Y_w=(n-k)\times1^{k-1}$ with $w=s_1s_2\cdots s_{k-1}s_ns_{n-1}\cdots s_k$, e.g. for $\Gr(3,8)$,
$w=s_1s_2s_7s_6\cdots s_3$ and the Young diagram is given by $Y=5\times 1^2$,
\[
\young[5,1,1][8][$q$,$1$,$q$,$1$,$q$,$1$,$q$].
\]
The number of $q$'s in this hook diagram is $m$, i.e. it has the weight $q^m$,
and the length (total number of boxes) of the diagram is $n-1=2m-1$.
\end{itemize}
Then for those Young diagrams in the graph $\mathsf{G}(k,n)$, the arrows coming in or going out are just $\rightarrow$, not $\Rightarrow$.  That is, $\eta(w)$ 
in the $q$-weighted Schubert variety $\langle Y_w:\eta(w)\rangle$ changes under those nontrivial $W$-actions, and those $w$ contribute the sum.
 One can also show that those are only ones which contribute the sum.  Then for example, in the case $(k,n)=(2j,2m)$ in the case (i), one can see that the polynomial $p(q)$ agrees with $|\Gr(j,m,\mathbb{F}_{q^2})|$ in \eqref{Qnumber}, which is the formula in (i).  Other cases in (i) can be obtained in the similar manner.

 For the case (ii), i.e. $(k,n)=(2j+1,2m)$, there are two types of Young diagrams
 which contribute the sum $p(q)$. The Young diagrams having
just  type (i) give the second term  $-\left[\begin{matrix}m-1\\j\end{matrix}\right]_{q^2}$ in $p(q)$ ($-$ sign is due to $(-1)^{k(n-k)}$), and 
those with the hook diagram give $q^m\left[\begin{matrix}m-1\\j\end{matrix}\right]_{q^2}$, the first term in $p(q)$.
\end{proof}

One can also define the polynomial $p^*(q)$ using $\eta^*(w)$ by
\[
p^*(q)=(-1)^{k(n-k)}\sum_{w\in S_n^{(k)}}(-1)^{\ell(w)}q^{{\eta^*(w)}}.
\]
This polynomial will be important for the non-orientable case of $\Gr(k,n)$
with $n=$odd, and we have:
\begin{Proposition}\label{q-poly}
For $\Gr(k,2m+1)$, the polynomial $p^*(q)$ takes the form,
\[
p^*(q)=q^s\left[\begin{matrix}m\\j\end{matrix}\right]_{q^2}=q^s\,p(q),
\]
where $s=j$ if $k=2j$, and $s=m-j$ if $k=2j+1$.
\end{Proposition} 


We then have the following theorem showing that $p(q)$ is related to the $\mathbb{F}_q$ points on $\Gr(k,n)$:
\begin{Theorem}  \label{points}
We assume $\sqrt{-1}\in\mathbb{F}_q$. Then we have
\[
|\Gr(k,n)_{\mathbb{F}_q}|= q^{r}p(q)\quad\text{with}\quad r=k(n-k)-{\rm deg}(p(q)).
\]
\end{Theorem}

In the following subsection, we give an outline of the proof based on
the action of the Hecke algebra and the Frobenius eigenvalues in etale
cohomology (the Weil \'etale cohomology theory).

\begin{Example} Let us some explicit examples:

\begin{itemize}
\item[(a)]$ \quad|\Gr(1,2m)_{\mathbb{F}_q}|=|(\mathbb{R}P^{2m-1})_{\mathbb{F}_q}|=q^{m-1}(q^m-1),$\item[(b)]$\quad
|\Gr(1,2m+1)_{\mathbb{F}_q}|=|(\mathbb{R}P^{2m})_{\mathbb{F}_q}|=q^{2m},$
\item[(c)] $\quad|\Gr(2,4)_{\mathbb{F}_q}|=q^2(1+q^2),$
\item[(d)]$\quad|\Gr(3,6)_{\mathbb{F}_q}|=q^4(1+q^2)(q^3-1).$
\end{itemize}
\end{Example}

\subsection{Outline of the proof of Theorem \ref{points}}
We recall that the number of points  $| \Gr(k,n)_{{\mathbb{F}}_q} |$ is given by a Lefshetz fixed-point Theorem applied to  the Frobenius map, $Fr:x\to x^q$ on
$\Gr(k,n)_{\bar{\mathbb{F}}_q}$ with the algebraically closed field $\bar{\mathbb{F}}_q$.  That is, we have the formula relating $| \Gr(k,n)_{{\mathbb{F}}_q} |$ to the Frobenius eigenvalues,
\[
\sum_{s=0}^{k(n-k)}(-1)^s {\rm Tr}\left((Fr) |_{H_c^s( \Gr(k,n)_{\overline{\mathbb{F}}_q}
; {\mathbb Q}_l)}\right)=| \Gr(k,n)_{{\mathbb{F}}_q} |,
\]
where $H_c^s(X,\mathbb{Q}_l)$ is the \'etale cohomology with compact support
of $X$ with the values in the $l$-adic number field $\mathbb{Q}_l$.

We also have the following Proposition which corresponds to Proposition 6.1 in \cite{casian:06}:
\begin{Proposition}\label{powersq} The induced action of the Frobenius map on  $H^s (\Gr(k,n)_{\bar {\mathbb F}_q},  {\mathbb Q}_l)$ with the $l$-adic number field ${\mathbb{Q}}_l$ has the eigenvalues
of the form $q^i$ where $i$ is an integer.
\end{Proposition}
By Definition \ref{polyQ} of the polynomial $p(q)$, Proposition \ref{powersq} then implies
\[
\sum_{s=0}^{k(n-k)}(-1)^s {\rm Tr}\left((Fr) |_{H^s( \Gr(k,n)_{\overline{\mathbb{F}}_q};
 {\mathbb Q}_l)}\right)=(-1)^{k(n-k)}p(q)
\]

Then to compute $|\Gr(k,n)_{\mathbb{F}_q}|$, we simply have to make a replacement in the  polynomial $p(q)$ corresponding to the {\em dual } of  $H^{s}_c( \Gr(k,n)_{\overline{\mathbb{F}}_q} ;  {\mathbb{Q}}_l)$.
Duality has the effect of replacing $q^{-1}$ instead of $q$ to take into account the action of Frobenius on the dual space. Then there is the additional  standard shift by $q^{k(n-k)}$ because of the Frobenius eigenvalues on  stalks of the coefficient
$\mathbb{V}$, the dual of a constant sheaf in the Poincar\'e duality formula. Thus we have $ q^{k(n-k)}p(q^{-1})$. These observations  leads to the proof of Theorem \ref{points}:
 
  We need to consider the cases with $n=$even (orientable case) and $n=$odd (non-orientable case):
\begin{enumerate}
\item Assume that $n$ is even.
By  the Poincar\'e duality, $H_c^{2k(n-k)-s}(\Gr(k,n)_{\overline{\mathbb{F}}_q} ; {\mathbb Q}_l)$ is the dual of $H^{s}( \Gr(k,n)_{\overline{\mathbb{F}}_q} ;  {\mathbb Q}_l )$.  
As has been already discussed above, this simply corresponds  to considering the polynomial  $ q^{k(n-k)}p(q^{-1})$.  

We now use standard properties of the  $q$ deformation of binomial coefficients.  
We note that  if $D={\rm deg}(p(q))$ then $q^Dp(q^{-1})=(-1)^{k(n-k)}p(q)$. We obtain
\[
\sum_{s=0}^{k(n-k)}(-1)^s {\rm Tr}\left((Fr)|_{H_c^s( \Gr(k,n)_{\overline{\mathbb{F}}_q};
 {\mathbb Q}_l)}\right)=| \Gr(k,n)_{{\mathbb {F}}_q} |=q^{k(n-k)-D}p(q).
\]

\item Assume that $n$ is odd, say $n=2m+1$.  There are two cases, i.e. $k=2j$ and $k=2j+1$.
By  the Poincar\'e-Verdier duality (see IX.4 and VI.3 in \cite{iversen}), $H_c^{2k(n-k)-s}(  \Gr(k,n)_{\overline{\mathbb{F}}_q} ;  {\mathbb Q}_l)$ is the dual of
$H^{s}( \Gr(k,n)_{\overline{\mathbb{F}}_q} ;  {\mathcal L} )$ where ${\mathcal L}$ is a twisted local system. 
We then have $|\Gr(k,n)_{\mathbb{F}_q}|=q^{k(n-k)}p^*(q^{-1})=q^{k(n-k)-D^*}p^*(q)$ with
$D^*=\text{deg}(p^*(q))$. Here the $q$-shifted polynomial
$p^*(q)$ are given in Proposition \ref{q-poly}, i.e.
$p^*(q)=q^sp(q)$ with $s=j$ if $k=2j$, and $s=m-j$ if $k=2j+1$.
Then we have $D^*=s+D$, and obtain
 $|\Gr(k,n)_{\mathbb{F}_q}|=q^{k(n-k)-D^*}p^*(q)=q^{k(n-k)-D}p(q)$.
\end{enumerate}

\section{The Poincar\'e polynomials}\label{Poincare}

 For $\Gr(k,n)$, the Poincar\'e polynomial, denoted by $P_{(k,n)}$, is given by
\[
P_{(k,n)}(t)=\sum_{i=1}^{k(n-k)}\text{dim}(H^i(\Gr(k,n),\Real))\,t^i.
\]
Here we show that  $P_{(k,n)}(t)$ can be obtained from the polynomial $p(q)$ defined in the previous section.
 \begin{Theorem} \label{B} 
 The Poincar\'e polynomial for $\Gr(k,n)$ is given as follows:
 \begin{itemize}
   \item [(i)]  If   $(k,n)$ equals $(2j, 2m), (2j, 2m+1)$ or $(2j+1, 2m+1)$, we have
   \[
   P_{(k,n)}(t)=p(t^2)= \left[\begin{matrix} m \\ j \end{matrix}\right]_{t^4} .
 \]
   \item [(ii)]  If $(k,n)=(2j+1,2m)$, we have
 \[
 P_{(k,n)}(t)= (t^{2m-1}+1 ) \left[\begin{matrix} m-1 \\ j \end{matrix}\right]_{t^4}.
 \]
 where $p(q)=(q^m-1)\left[\begin{matrix}m-1\\j\end{matrix}\right]_{q^2}$, that is, replace $(q^{m}-1)$
 by $(t^{2m-1}+1)$ and set $q=t^2$.
\end{itemize}  
  \end{Theorem}
Note in particular that for the case (i), we have 
  \[
  P_{(k,n)}(t)=P^{\mathbb{C}}_{(\lfloor{k}/{2}\rfloor,\lfloor{n}/{2}\rfloor)}(t^2).
  \]

These polynomials then lead to the well-known formulas of the Euler characteristic for $\Gr(k,n)$,
that is, we have:
\begin{Corollary}
The Euler characteristic $\chi_E(\Gr(k,n))$ has the following form.
\begin{itemize}
\item[(i)] If $(k,n)$ equals $(2j,2m), (2j,2m+1)$ or $(2j+1,2m+1)$, then
\[
\chi_E(\Gr(k,n))=P_{(k,n)}(-1)=\binom{m}{j}.
\]
\item[(ii)] If $(k,n)=(2j+1,2m)$, then
\[
\chi_E(\Gr(k,n))=P_{(k,n)}(-1)=0.
\]
\end{itemize}
\end{Corollary}

\subsection{The Poincar\'e-Verdier duality}
When the Grassmannian $\Gr(k,n)$ is orientable, $n$ is even, we have
the {\it Poincar\'e duality}
\[
H_j(\Gr(k,n),\mathbb{Z})=H^{k(n-k)-j}(\Gr(k,n),\mathbb{Z}).
\]
In the non-orientable case, i.e. $n$ is odd, 
we define the incidence graph with twisted coefficients $\mathcal{L}$,
denoted by $\mathsf{G}(k,n)^*$, that is, this graph is obtained by the Young diagrams with $q$-shifted arrangement.
The  graph then gives the cohomology of twisted coefficient,
\[
H^*(\Gr(k,n),\mathcal{L}).
\]
which then gives the homology of the non-orientable  Grassmannian $\Gr(k,n)$, via
the {\it Poincar\'e-Verdier duality} (see e.g. Section IX.4 and VI.3 in \cite{iversen}),
\[
H_j(\Gr(k,n),\mathbb{Z})=H^{k(n-k)-j}(\Gr(k,n),\mathcal{L}).
\]

\begin{Example}
The graph $\mathsf{G}(2,5)^*$ is given by
\[
\begin{matrix}
\emptyset  & {\Longrightarrow}  &     \young[1][8][$1$]      & {\longrightarrow}   &     \young[2][8][$1$,$q$]        &{\Longrightarrow} & \young[3][8][$1$,$
q$,$1$]    \\[-2.5ex]\\
             &                                &{\downarrow }    &                               &       {\downarrow } &                          &{\downarrow} \\[-2ex]\\
            &                               &    \young[1,1][8][$1$,$q$]        & {\longrightarrow }   &  \young[2,1][8][$1$,$q$,$q$]           & {\Longrightarrow} & \young[3,1][8][$1$,$q$,$1$,$q$]   \\[-2ex]  \\
            &                           &                                      &                                &  {\Downarrow}   &                              &{\Downarrow} \\[-2ex]\\
            &                            &                                   &                                &       \young[2,2][8][$1$,$q$,$q$,$1$]       &{\Longrightarrow} &\young[3,2][8][$1$,$q$,$1$,$q$,$1$] \\[-2ex]\\
            &                          &                                     &                               &                                        &                         &{\downarrow} \\[-2ex]\\
            &                          &                                  &                                  &                                       &                         & \young[3,3][8][$1$,$q$,$1$,$q$,$1$,$q$]
  \end{matrix}
\]
which is obtained by exchanging $q\leftrightarrow 1$ in the $q$-weighted Schubert varieties
in the incidence graph $\mathsf{G}(k,n)$.  The cohomology obtained from this graph is 
\[\begin{array}{llllll}
H^0(\Gr(2,5),\mathcal{L})=0, \hskip 1.5cm    & H^4(\Gr(2,5),\mathcal{L})=\mathbb{Z}_2, \\[1.0ex]
H^1(\Gr(2,5),\mathcal{L})=\mathbb{Z}_2,                                      & H^5(\Gr(2,5),\mathcal{L})=\mathbb{Z}_2, \\[1.0ex]
H^2(\Gr(2,5),\mathcal{L})=\mathbb{Z},                        &H^6(\Gr(2,5),\mathcal{L})=\mathbb{Z}.   \\[1.0ex]
H^3(\Gr(2,5),\mathcal{L})=\mathbb{Z}_2,                        & {}                                         
\end{array}
\]
Then the homology group $H_*(\Gr(2,5),\mathbb{Z})$ via
the Poincar\'e-Verdier duality gives
\[
H_j(\Gr(2,5),\mathbb{Z})=H^{6-j}(\Gr(2,5),\mathcal{L}) \quad \text{for}\quad j=0,1,\ldots,6.
\]
Notice that the homology generators are the Pontryagin classes in $H_0$ and $H_4$.
\end{Example}

\section{Conjecture on the ring structure of the cohomology}

Recall (see e.g. \cite{MS:74}) that the cohomology ring of the {\it complex} Grassmannian $\Gr(k,n,\mathbb{C})$ is given by
\[
H^*(\Gr(k,n,\mathbb{C}),\mathbb{R})\cong \frac{\mathbb{R}[c_1,\ldots,c_{n-k},\bar{c}_1,\ldots,\bar{c}_k]}{\{c\cdot\bar{c}=1\}},
\]
where $c=1+c_1+\cdots+c_{n-k}$ and $\bar{c}=1+\bar{c}_1+\cdots+\bar{c}_k$ with the Chern classes $c_j\in H^{2j}(\Gr(k,n,\mathbb{C},\mathbb{R})$.

\medskip

The cohomology ring of the classifying space $BO(k)=\Gr(k,\infty)$ is also known, and it is given by
\[
H^*(BO(k),\mathbb{R})\cong \mathbb{R}[p_1,\ldots,p_{\lfloor \frac{k}{2}\rfloor}],
\]
where the generators $p_j$ of the ring are given by the Pontryagin classes $p_j\in H^{4j}(BO(k),\mathbb{R})$.

Then it is natural to make the following conjectures:
\begin{Conjecture}
The homology ring $H^*(\Gr(k,n),\Real)$ is given by the following:
\begin{itemize}
\item[(i)] If $(k,n)$ equals $(2j,2m), (2j,2m+1)$ or $(2j+1,2m+1)$, then 
\[
H^*(\Gr(k,n),\mathbb{R})\cong \frac{\mathbb{R}[p_1,\ldots, p_{m-j},\bar{p}_1,\ldots,\bar{p}_j]}{\{p\cdot \bar{p}=1\}},
\]
where $p=1+p_1+\cdots+p_{m-j}$ and $\bar{p}=1+\bar{p}_1+\cdots+\bar{p}_j$ with the Pontryagin classes $p_j\in H^{4j}(\Gr(k,n),\mathbb{R})$.
\item[(ii)] If $(k,n)=(2j+1,2m+2)$, then
\[
H^*(\Gr(k,n),\mathbb{R})\cong \frac{\mathbb{R}[p_1,\ldots,p_{m-j},\bar{p}_1,\ldots,\bar{p}_j,r]}{\{p\cdot \bar{p}=1, r^2\}},
\]
where $p_j\in H^{4j}(\Gr(k,n),\mathbb{R})$ and the element $r$ corresponds to the Schubert class $\sigma_w$ with the hook diagram $Y_w=(n-k)\times 1^{k-1}$.
\end{itemize}
\end{Conjecture}



\bibliographystyle{amsalpha}

\end{document}